\documentclass{amsart}
\usepackage{amssymb, latexsym}

\newtheorem{theorem}{Theorem}
\newtheorem{lemma}{Lemma}
\newtheorem{remark}{Remark}
\newtheorem{definition}{Definition}
\newtheorem{example}{Example}

\begin{document}

\title{Weak Width of Subgroups}
\author{Rita Gitik}
\email{ritagtk@umich.edu}
\address{ Department of Mathematics \\ University of Michigan \\ Ann Arbor, MI, 48109   }  

\date{\today}

\begin{abstract} We say that the weak width of an infinite subgroup $H$ of $G$ in $G$ is $n$ if there exists a collection of $n$ strongly essentially distinct conjugates

$\{ H, g_1^{-1} H g_1,\cdots, g_{n-1}^{-1} H g_{n-1}  \}$ of  $H$ in $G$ such that the intersection $H \cap g_i^{-1} H g_i$ is infinite for all $1 \leq i \leq n-1$  and $n$ is maximal possible. We prove that a quasiconvex subgroup of a negatively curved group has finite weak width in the ambient group.
We also give examples demonstrating that height, width, and weak width are different invariants of a subgroup.
\end{abstract}
\maketitle

\section{Introduction}

A subgroup $H$ of $G$ is malnormal in $G$ if for any $g \in G$ such that $g \notin H$ the intersection $H \cap g^{-1} H g$ is trivial.
Most subgroups are neither normal nor malnormal, so the study of the intersection pattern of conjugates of a subgroup is an interesting problem. It is closely connected to the study of the behavior of different lifts of subspaces of topological spaces in  covering spaces. 

Malnormality of a subgroup has been generalized in different ways. One of them, namely the height, introduced  in \cite{G-R}, has been used by Agol in \cite{Ag}  and \cite{A-G-M} in his proof of  Thurston's conjecture that $3$-manifolds are virtual bundles. In this
paper we introduce yet another generalization of malnormality. It is a new invariant of the conjugacy class of a subgroup $H$ of $G$, which we call the weak width of a subgroup. Like malnormality, the weak width measures only the cardinality of the intersections of $H$ with its conjugates in $G$. In section $4$ we prove that quasiconvex subgroups of negatively curved groups have finite weak width, which might simplify Agol's proof. In section $2$ we review the definitions and the basic properties of the width and the height of a subgroup. In section $3$ we give examples showing that height, width, and weak width are different invariants of a subgroup.

\begin{remark} \label{R:ShortRep}
 Note that if $g_i \in Hg_jH$, hence $g_i=h_1g_jh_2$ with $h_1$ and $h_2$ in 
$H$, then  $H \cap g_i^{-1}Hg_i = H \cap (h_1g_jh_2)^{-1}H(h_1g_jh_2)= H \cap (h_2^{-1}g_j^{-1}Hg_jh_2)=
h_2^{-1}(H \cap g_0^{-1}Hg_0)h_2$. So the cardinality of the set $H \cap g_i^{-1}Hg_i$ is equal to the cardinality of the set $H \cap g_j^{-1}Hg_j$. 
\end{remark}

Remark \ref{R:ShortRep} motivates the following definitions.

\begin{definition}
Let $H$ be a subgroup of a group $G$. We say that the elements $\{g_i| 1 \leq i \leq n \}$ of $G$ are strongly $H$-essentially distinct if $Hg_iH \neq Hg_jH$ for $i \neq j$. Conjugates $g_i^{-1} H g_i$ of $H$ by strongly $H$-essentially distinct elements are called strongly essentially distinct conjugates.
\end{definition}

\begin{definition}
We say that the weak width of an infinite subgroup $H$ of $G$ in $G$, denoted $WeakWidth(H,G)$, is $n$ if there exists a collection of $n$ strongly essentially distinct conjugates $\{ H, g_1^{-1} H g_1,\cdots, g_{n-1}^{-1} H g_{n-1} \}$ of  $H$ in $G$ such that the intersection $H \cap g_i^{-1} H g_i$ is infinite for all $1 \leq i \leq n-1$  and $n$ is maximal possible. We define the weak width of a finite subgroup of $G$ to be $0$.
\end{definition}

Note that if $WeakWidth(H,G)=n$, then in any set of $n+1$ strongly essentially distinct conjugates  $ \{H, g_1^{-1} H g_1,\cdots, g_n^{-1} H g_n \}$ of $H$ in $G$ there exists an element
$g_i^{-1} H g_i$ which has finite intersection with $H$.

\section{Height and width.}
The following definitions were introduced in \cite{G-R} and \cite{G-M-R-S}.

\begin{definition}
Let $H$ be a subgroup of a group $G$. We say that the elements $\{g_i| 1 \leq i \leq n \}$ of $G$ are $H$-essentially distinct if $Hg_i \neq Hg_j$ for $i \neq j$. Conjugates $g_i^{-1} H g_i$ of $H$ by $H$-essentially distinct elements are called essentially distinct conjugates.
\end{definition}

If $g_i$ and $g_j$ are not $H$-essentially distinct, then $g_i^{-1} H g_i = g_j^{-1} H g_j$, hence it is interesting to investigate the intersections of the conjugates of $H$ only if they are $H$-essentially distinct.
However, essentially distinct conjugates need not be distinct. For example, let $G=<a_1, a_2|a_1a_2=a_2a_1>$ be a free abelian group of rank $2$ and let $H=<a_1>$ be a subgroup of $G$. The conjugates $a_2^{-1} H a_2$ and $H$ are essentially distinct, but  $a_2^{-1} H a_2 =H$.

\begin{definition}
We say that the height of an infinite subgroup $H$ of $G$ in $G$, denoted by $Height(H,G)$, is $n$  if there exists a collection of $n$ essentially distinct conjugates of $H$ in $G$ such that the intersection of all the elements of the collection is infinite and $n$ is maximal possible.  We define the height of a finite subgroup of $G$ to be $0$.
\end{definition}

Note that if $Height(H,G) =n$ then the intersection of any set of $n+1$ essentially distinct conjugates of $H$ in $G$ is finite.
It was shown in \cite{G-R} that subgroups of negatively curved groups have finite height in the ambient group.

\begin{definition}
We say that the width of an infinite subgroup $H$ of $G$ in $G$, denoted by $Width(H,G)$, is $n$ if there exists a collection of $n$ essentially distinct conjugates of $H$ in $G$ such that the intersection of any two elements of the collection is infinite and $n$ is maximal possible.  We define the width of a finite subgroup of $G$ to be $0$.
\end{definition}

Note that if $Width(H,G)=n$ then in any set of $n+1$ essentially distinct conjugates of $H$ in $G$ there exist two elements with finite intersection. It was shown in \cite{G-M-R-S} and, later, in \cite{H-W} that quasiconvex subgroups of negatively curved groups have finite width in the ambient group.
\\
It follows from the above definitions that $Width(H,G)$ and $Height(H,G)$ are invariants of the conjugacy class of $H$ in $G$.

Note also that $Height(H, G) \leq Width(H, G)$, however, it is not clear if there is any relationship between  $WeakWidth(H, G)$ and $Width(H, G)$. 

Infinite normal subgroups of infinite index have infinite height, width, and weak width in the ambient group. More generally, if an infinite subgroup has infinite index in its  normalizer, then the
 subgroup has infinite height, width, and weak width in the ambient group.
 
 If $G$ is torsion-free and $H$ is infinite, then $H$ is malnormal in $G$ if and only if $Height(H,G)= Width(H,G)=WeakWidth(H,G)=1$. 

\section{Examples}

The following examples demonstrate that $WeakWidth(H,G)$, $Width(H,G)$, and $Height(H,G)$ are distinct invariants of the conjugacy class of $H$ in $G$. 

Let $X$  be a set and let $X^* = \{x,x^{-1} |x \in X \}$, where for $x \in X$ we define $(x^{-1})^{-1} =x$. Denote the equality of two  words in $X^*$ by $" \equiv "$.  

\begin{example}\label{E:WWneqW}
Let $F$ be a free group of rank $4$ generated by the elements $x_1,x_2, x_3, x_{4}$, let $G = <F, t|t^4 =1, t^{-1} x_i t = x_{(i+1)mod4}| 1 \le i\le 4>$, and let $H_1=<x_1,x_2>$. We claim that $WeakWidth(H_1,G)=3$, but $Height(H_1,G)=Width(H_1, G)=2$.

In order to prove the claim we will list all essentially distinct and all strongly essentially distinct conjugates of $H_1$ in $G$ which have non-trivial intersection with $H_1$.

Let $H_i = <x_i, x_{(i+1)mod4}| 1\le i \le 4 > = t^{(-i+1)}H_1 t^{(i-1)}$ be conjugates of $H_1$ in $G$. As $t^i \notin F$ for $i \not\equiv 0\pmod{4}$, the conjugates $\{H_i| 1\le i \le 4 \}$ are strongly essentially distinct. Note that $H_2 \cap H_1 = <x_2> ,H_4 \cap H_1 =<x_1>, H_3 \cap H_1 = <1>$, and $H_2 \cap H_4 = <1>$.  Hence $WeakWidth(H_1,G) \ge 3$,  $Height(H_1,G) \ge 2$, and $Width(H_1, G) \ge 2$.
        
In order to determine how other conjugates of $H_1$ intersect, we 
consider $g \in G$  such that the intersection $g^{-1}H_1g \cap H_1$ is non-trivial. As we are interested only in essentially distinct conjugates of $H_1$, we can assume that $g$ is a shortest element in the coset $H_1g$.

 As $t$ normalizes $F$, we have $g=wt^k, 0 \le k \le 3$, with $w$ a reduced word in $F$. If $w$ is trivial, then $g^{-1}H_1g= t^{-k}H_1t^k = H_{1+k}$, and the intersection pattern of the subgroups $\{H_i|1 \le i \le 4\}$ is described above.
 
 If $w$ is non-trivial, let $v \in H_1$ be a 
 non-trivial reduced word such that $g^{-1}vg = (t^{-k}w^{-1}) v (w t^k)  \in H_1$. Then $w^{-1} v w \in t^k H_1 t^{-k}  =t^{-(4-k)}H_1t^{4-k}  = H_{(1-k)mod4}$. As $w$ and $v$ are reduced words in a free group $F$, there exist decompositions $w \equiv w_1w_2$ and $v \equiv w_1v_0w_1^{-1}$ (where $\equiv$ denotes  equality of words) with 
 
 $w^{-1} v w = (w_2^{-1}w_1^{-1})(w_1v_0w_1^{-1})(w_1w_2) = w_2^{-1}v_0w_2$, where $w_2^{-1}v_0w_2$ is a reduced word in $H_{(1-k)mod4}$. Then $v_0 \in H_{(1-k)mod4}$ and  $w_2 \in H_{(1-k)mod4}$. As  $v \in H_1$, it follows that $w_1 \in H_1$ and $v_0 \in H_1$.
However, as $g=wt^k =w_1w_2t^k$ is shortest in the coset 
 $H_1g$, $w_1$ should be trivial. Hence $w=w_2 \in H_{(1-k)mod4}$. As a non-trivial word $v_0$ belongs to $H_1 \cap H_{(1-k)mod4}$, it follows that $(1-k)\pmod{4}$ is equal to either $1,2$ or $4$. Hence if $(1-k) \pmod4 \equiv3$, so $k=2$, then for any $r \in F$ the intersection $ (rt^2)^{-1}H_1(rt^2) \cap H_1$ is trivial. 

If $(1-k)\pmod{4} \equiv1$ then $w=w_2 \in H_1$, contradicting again the fact that $g$ is shortest in the coset $H_1g$. Hence either $(1-k)\pmod{4} \equiv 2$ and $k=3$, or $(1-k)\pmod{4} \equiv4$ and $k=1$.

If $k=3$, then $g=wt^3$ with $w \in H_2$. Note that the essentially distinct elements of the infinite collection of the conjugates $\{(wt^3)^{-1}H_1(wt^3)| w \in H_2 \}$ intersect each other trivially. Indeed, consider $w_0 \in H_2$ and $w \in H_2$ such that the intersection $(t^{-3}w^{-1})H_1(wt^3) \cap (t^{-3}{w_0}^{-1})H_1(w_0t^3)$  is non-trivial. Then the intersection $H_1 \cap (w_0t^3)(t^{-3}w^{-1})H_1(wt^3)(t^{-3} {w_0}^{-1})$ is  non-trivial. As $H_1$ is malnormal in $F$, it follows that  $w_0w^{-1}=(w_0t^3)(t^{-3}w^{-1}) \in H_1$, so the conjugates 
$(t^{-3}w^{-1})H_1(wt^3)$ and $(t^{-3}{w_0}^{-1})H_1(w_0t^3)$ are not essentially distinct. Therefore the family of the conjugates
$\{(wt^3)^{-1}H_1(wt^3)| w \in H_2 \}$
does not contribute to $Width(H_1,G)$.

Similarly, if $k=1$, hence $g=ut$ with $u \in H_4$, the  essentially distinct elements of the infinite collections of the conjugates $\{(ut)^{-1}H_1(ut)| u \in H_4 \}$ intersect each other trivially.

Also for $w \in H_2$ and $u \in H_4$ the intersection $(t^{-3}w^{-1})H_1(wt^3) \cap (t^{-1}u^{-1})H_1(ut)$ is trivial. Indeed, the cardinality of that intersection is equal to the cardinality of the intersection
$(ut)(t^{-3}w^{-1})H_1(wt^3)(t^{-1}u^{-1}) \cap H_1$. However, $(wt^3)(t^{-1}u^{-1})=wt^2u^{-1}=(w(t^2u^{-1}t^{-2})t^2 = rt^2$ with $r \in F$, and we have mentioned above that for all $r \in F$ the intersection $ (rt^2)^{-1}H_1(rt^2) \cap H_1$ is trivial. So the infinite family of conjugates $\{(ut)^{-1}H_1(ut)| u \in H_4 \}$
does not contribute to $Width(H_1,G)$, therefore $Height(H_1,G)= Width(H_1,G)=2$. 

Note that for any $w \in H_2$,  $wt^3=t^3(t^{-3}wt^3) \in 
t^3H_1 \subseteq H_1t^3H_1$, hence all the elements   $\{(wt^3)| w \in H_2 \}$  are strongly $H_1$-equivalent to $t^3$, so the conjugates of $H_1$ by those elements do not contribute to the weak width of $H_1$. Similarly, all the elements $\{(ut)^{-1}| u \in H_4 \}$ are strongly $H_1$-equivalent to $t$, so the conjugates of $H_1$ by those elements  do not contribute to the weak width of $H_1$ either.
Therefore, $WeakWidth(H,G)=3$.
 
$\hfill\square$
\end{example}

\begin{example}\label{WneH}
Let $G$ be as in Example \ref{E:WWneqW}, and let $L_1=<x_1,x_2, x_3>$. We claim that $WeakWidth(L_1,G)=Width(L_1, G)=4$, but $Height(L_1,G)=3$.

Let $L_i = <x_i, x_{(i+1)mod4}, x_{(i+2)mod4}| 1\le i \le 4 > = t^{(-i+1)}L_1 t^{(i-1)}$ be conjugates of $L_1$ in $G$. As $t^i \notin F$ for $i \not\equiv 0\pmod{4}$, the conjugates $\{L_i| 1\le i \le 4 \}$ are strongly essentially distinct. 
By observation, the elements of the set $\{ L_i| 1 \le i \le 4 \}$ have infinite pairwise intersections, hence
$WeakWidth(L_1,G) \ge 4$ and $Width(L_1, G) \ge 4$. Also the intersection $\bigcap\limits_{i =1}^3 L_i$ is infinite, so $Height(L_1,G) \ge 3$. Note also that
the intersection $\bigcap\limits_{i =1}^4 L_i$ is trivial. 

Using the same argument as in Example \ref{E:WWneqW} we can show that there are only three families of $L_1$-essentially distinct  elements in $G$ such that the conjugates of $L_1$ by these elements intersect $L_1$ non-trivially. They are $\{(wt^3)| w \in L_2 \},
\{ ut| u \in L_4 \}$, and $ \{ st^2| s \in L_3 \}$. Just as in Example \ref{E:WWneqW}, the malnormality of $L_1$ in $F$ implies that the essentially distinct conjugates in each family intersect each other trivially, hence $Width(L_1, G)=4$. Also as in Example \ref{E:WWneqW} these elements are strongly $L_1$-essentially equivalent to $t^3, t$, and $t^2$, respectively, so $WeakWidth(L_1,G)=4$.

Suppose $Height(H,G) \ge 4$. Then there are $3$ essentially distinct conjugates $M_2, M_3$, and $M_4$ of $L_1$ such that the intersection $L_1 \cap (\bigcap\limits_{i=2}^4 M_i)$ is infinite. The preceding paragraph implies that the $M_i$'s must come one from each of  
the families of conjugates of $L_1$ described above, i.e. $M_2, M_3$ and $M_4$ are conjugates of $L_1$ by $wt^3, ut$, and $st^2$ respectively, with $w \in L_2, u \in L_4$, and $s \in L_3$. Let $h_1, h_2, h_3$, and $h_4$ in $L_1$ be such that $h_4 =t^{-3}w^{-1}h_1wt^3=t^{-2}s^{-1}h_2st^2=t^{-1}u^{-1}h_3ut$. Note that $t^{-3}wt^3 \in L_1, t^{-3}h_1t^3 \in L_4, t^{-2}st^2 \in L_1, t^{-2}h_2t^2 \in L_3, t^{-1}ut \in L_1$, and $t^{-1}h_3t \in L_2$. Then $t^{-3}w^{-1}h_1wt^3=r^{-1}_1q_1r_1$ with $r_1 \in L_1$ and $q_1 \in L_4$, $t^{-2}s^{-1}h_2st^2=r_2^{-1}q_2r_2$ with $r_2 \in L_1$ and $q_2 \in L_3$, and $t^{-1}u^{-1}h_3ut = r_3^{-1}q_3r_3$ with $r_3 \in L_1$ and $q_3 \in L_2$. 
As  $r^{-1}_1q_1r_1 = r_2^{-1}q_2r_2 = r_3^{-1}q_3r_3$,
 it follows that $q_2 = l_1^{-1}q_1l_1=l_2^{-1}q_3l_2$ with $l_1$ and $l_2 $ in $L_1$. We can assume that all the words $l_1, l_2, q_1, q_2$, and $q_3$ are reduced. Then, as in Example \ref{E:WWneqW}, there exist decompositions
$l_1 \equiv p_1p_2$ and $q_1 \equiv p_1q'_1p_1^{-1}$ such that $q_2 = (p_1p_2)^{-1}(p_1q'_1p_1^{-1})(p_1p_2)=p_2^{-1}q'_1p_2$, and $p_2^{-1}q'_1p_2$ is a reduced word in $F$.

 As 
$r^{-1}_1q_1r_1 = r_2^{-1}q_2r_2 = r_3^{-1}q_3r_3 = h_4 \in L_1$, it follows that $q_1 \in L_1 \cap L_4 =<x_1, x_2>, q_2 \in L_1 \cap L_3 = <x_1, x_3>$, and $q_3 \in L_1 \cap L_2 = <x_2, x_3>$. As $q_1 \in <x_1, x_2>$ and $q_2 \in <x_1, x_3>$, it follows that $q'_1 = x_1^n$ for $ n \in \textbf{N}$. 

Simirlarly, there exist decompositions $l_2 \equiv c_1c_2$ and $q_3 \equiv c_1q'_3c_1^{-1}$ such that $q_2 = (c_1c_2)^{-1}(c_1q'_3c_1^{-1})(c_1c_2)=c_2^{-1}q'_3c_2$, and $c_2^{-1}q'_3c_2$ is a reduced word in $F$. As $q_3 \in <x_2, x_3>$ and $q_2 \in <x_1, x_3>$, it follows that $q'_3 = x_3^m$ for $m \in \textbf{N}$. Then a conjugate of $q'_1 = x_1^n$ is equal to a conjugate of $q'_3 = x_3^m$ in a free group $F$. This can happen only if  $q'_1$ and $q'_3$ are trivial, hence $q_2$ is trivial. Therefore, the intersection of $L_1$ with all three families of conjugates is trivial, so $Height(L_1, G)=3$.

$\hfill\square$
\end{example}

\section{Quasiconvex subgroups of negatively curved groups have finite weak width}

We will use the following notation.

Let $G$ be a group generated by the set $X^*$.
As usual, we identify the word in $X^*$ with the corresponding element in $G$. Let $Cayley(G)$ be the Cayley graph of $G$ with respect to the generating set $X^*$. The set of vertices of $Cayley(G)$ is $G$,  the set of edges of $Cayley(G)$ is $G \times X^*$, and the edge $(g,x)$ joins the vertex $g$ to $gx$.

\begin{definition} 
The label of the path  
$p=(g,x_1)(gx_1,x_2)  \cdots (gx_1x_2 \cdots x_{n-1},x_n)$ 
in $Cayley(G)$ is the word $Lab (p) \equiv x_1 \cdots x_n $.
The length of the path $p$, denoted by $|p|$, is the number of edges forming it. The inverse of a path $p$ is denoted by $\bar p$, 
\end{definition}

\begin{remark}\label{R:ShiftQuasiconv}
Let $H$ be a $K$-quasiconvex subgroup of $G$, let $\eta$ be a geodesic in $Cayley(G)$ with $Lab(\eta) \in H$, and let $\eta'\eta''$ be any decomposition of $\eta$. There exists a path $c$ with $|c|\le K$ which begins at the terminal vertex of $\eta'$ such that $Lab (\eta'c) \in H$ and $Lab(\bar c \eta'') \in H$. Indeed, if $\eta$ begins (and , hence, ends) at an element of $H$ such $c$ exists by the definition of $K$-quasiconvexity. In the general case, we can find such $c$ using translation in $Cayley(G)$. 
\end{remark}

The following result was essentially proven in \cite{G-M-R-S}. We include a streamlined version of the proof.

\begin{theorem}\label{T:FiniteWeakWidth}
 If $H$ is a quasiconvex subgroup of a negatively 
curved  group $G$, then $WeakWidth(H,G)$ is finite. 
\end{theorem}
\begin{proof}
As $G$ is finitely generated, there exists a finite number $N$ of elements in $G$ of length not greater  than $2K+ 2 \delta$, hence there exist at most $N$ strongly $H$-essentially distinct elements $\{ g_i \in G \}$ such that the shortest representative of the double coset
 $Hg_iH$ is not longer than  $2K+ 2 \delta$. Then Lemma \ref{L:SmallIntersection} implies that the only strongly $H$-essentially distinct conjugates of 
$H$ which might have infinite intersection with $H$ are the conjugates of $H$ by the elements in the set $\{ g_i | 1 \le i \le N \}$. Therefore $WeakWidth(H,G) \le N$.
\end{proof}

\begin{lemma}\label{L:SmallIntersection}
Let $H$ be a $K$-quasiconvex  subgroup of a
$\delta$-negatively curved group $G$ and $g$ be an element in $G$. If every element of  
the double coset $HgH$ is longer than $2K + 2 \delta$, then the intersection
$H \cap g^{-1}Hg$ is finite.
\end{lemma}

\begin{proof}
Remark \ref{R:ShortRep} implies that it is sufficient to prove Lemma  \ref{L:SmallIntersection} for a shortest representative $g_0$ of the double coset $HgH$.

 We will show that all the elements in the intersection $H \cap g_0^{-1}Hg_0$ are shorter than $2K+8 \delta+2$, so that the intersection is finite, as required. 
 
Consider $h \in H \cap g_0^{-1} H g_0$. Let  $h_0 \in H$  be such that 
$h=g_0^{-1}h_0g_0$.
Let $p_1,p_{h_0}, p_2$ and  $p_h$ be geodesics in  $Cayley(G)$ such that $p_1p_{h_0}p_2 \bar p_h$ is a closed path, $p_1$ (hence also $p_h$) begins at $1$, $Lab(p_1)= g_0^{-1}, Lab(p_{h_0})= h_0, Lab(p_2)= g_0$, and $ Lab(p_h)= h = g_0^{-1}h_0g_0$. 

Let $v$ be a middle  vertex of $p_h$ and let $q$ be the initial subpath of $p_h$ ending at $v$. As $Lab(p_h) = h \in H$, Remark \ref{R:ShiftQuasiconv} implies that there exists a path $s$ with $|s| \le K$ which begins at $v$ such that $Lab(qs) \in H$. 
Let $t$ be a shortest path which begins at $v$ and ends at some vertex $w$ of $p_{h_0}$ and let $q'$ be the initial subpath
of $ p_{h_0}$ terminating at $w$.
As $Lab(p_{h_0}) = h_0 \in H$, Remark \ref{R:ShiftQuasiconv} implies that there exists a path $s'$ with $|s'| \le K$ which begins at $w$ such that $Lab(q's') \in H$. 
Then  $Lab(\bar p_1)= g_0 =Lab(q's')Lab(\bar s' \bar t s) Lab(\bar s \bar q)$.
As $Lab(q's') \in H$ and $Lab(\bar s \bar q)= Lab^{-1}(qs) \in H$, it follows that $Lab (s' \bar t s) \in
Hg_0H$.

As $HgH = Hg_0H$, the assumption of Lemma \ref{L:SmallIntersection} that any element of the double coset $HgH$ is longer than $2K + 2 \delta$ implies that $|s'\bar ts| > 2K + 2 \delta$. 
But then $|t| > 2K +2 \delta -|s|-|s'| > 2 \delta $, hence the distance from $v$ to $p_{h_0}$ is greater than $2 \delta$.

As $G$ is $\delta$-negatively curved, a side $p_h$ of the geodesic $4$-gon  $p_1p_{h_0}p_2 \bar p_h$ belongs to the $2 \delta$-neighborhood of the union of the other
three sides, so the above discussion implies that $v$ belongs to the $2\delta$-neighborhood of $p_1 \cup  p_2$. Assume that there exists a path $y$ of length less than $2 \delta$
which begins at a vertex $u$ of $p_1$ and ends at $v$. Consider the decomposition $p_1=p'_1p''_1$, where $p'_1$ ends at $u$. As
$g_0=Lab(\bar p_1)=Lab(\bar p''_1 ys) Lab (\bar s \bar q)$ and $Lab(\bar s \bar q) \in H$, it follows that $Lab(\bar p''_1 ys) \in g_0H$. As $g_0$ is a shortest representative of $g_0H$, it follows that  $|g_0|=|p'_1|+|p''_1| \le |\bar p''_1 ys|=|s| +|y| + |p''_1|$.  Hence   
$|p'_1| \le |s| +|y| \leq K +2 \delta$, so 
$|q| \le |p'_1| + |y| \le K +4\delta$. However, as $v$ is a middle point of $p_h$,  $|h|=|p_h| \leq 2|q|+1 <
2K+8\delta +2$.

Similarly, if $v$ belongs to the $2 \delta$-neighborhood of $p_2$, 
it follows that $|h| < 2K+8\delta +2$, proving Lemma \ref{L:SmallIntersection}.
\end{proof}

\section{Question}
Is their a simple relation between the width and the weak width?

\textbf{Acknowledgment}

The author would like to thank Shmuel Weinberger for his support.

\end{document}